\documentclass{amsart}

\makeatletter
\@namedef{subjclassname@2020}{%
  \textup{2020} Mathematics Subject Classification}
 
\makeatother 

\usepackage{amsthm,amssymb,amsfonts,latexsym,mathtools,thmtools,amsmath,lscape}
\usepackage[T1]{fontenc}
\usepackage{tikz-cd} 
\usepackage{enumitem} 
\usepackage{longtable}
\usepackage{multirow}
\usepackage{mathrsfs} 
\usepackage{hyperref} 
\usepackage{hyperref}
\usepackage[all]{xy}
\usepackage{booktabs}
\hypersetup{
    colorlinks=true,
    linkcolor=blue,
    filecolor=blue,      
    urlcolor=blue,
    linktocpage=true
}

\newtheorem{theorem}{Theorem}[section]

\newtheorem{proposition}[theorem]{Proposition}

\theoremstyle{definition}
\newtheorem{definition}[theorem]{Definition}

\newtheorem{remark}[theorem]{Remark}

\newtheorem{example}[theorem]{Example}

\theoremstyle{remark}

\numberwithin{equation}{section}

\begin{document}

\title[Differential smoothness of bi-quadratic algebras]{Differential smoothness of \\ bi-quadratic algebras with PBW basis}


\author{Andr\'es Rubiano}
\address{Universidad Nacional de Colombia - Sede Bogot\'a}
\curraddr{Campus Universitario}
\email{arubianos@unal.edu.co}
\address{Universidad ECCI}
\curraddr{Campus Universitario}
\email{arubianos@ecci.edu.co}
\thanks{}


\author{Armando Reyes}
\address{Universidad Nacional de Colombia - Sede Bogot\'a}
\curraddr{Campus Universitario}
\email{mareyesv@unal.edu.co}



\date{}

\dedicatory{Dedicated to the memory of our dear Professor Mikhail Malakhaltsev}

\begin{abstract} 

We investigate the differential smoothness of bi-quadratic algebras with PBW basis.
\end{abstract}

\maketitle


\section{Introduction}

The {\em theory of connections} in noncommutative geometry is a well-established framework (see Connes \cite{Connes1994} and Giachetta et al. \cite{Giachettaetal2005}). This theory considers a differential graded algebra $\Omega A = \bigoplus_{n = 0} \Omega^{n} A$ over a $\Bbbk$-algebra $A = \Omega^{0} A$ with $\Bbbk$ a field. A {\em connection} on a left $A$-module $M$ is defined as a $\Bbbk$-linear map $\nabla^{0} : M \to \Omega^{1} A \otimes_A M$ satisfying the Leibniz rule: $\nabla^0(am) = da \otimes_A m + a\nabla^{0}(m)$ for all $a \in A$ and $m \in M$. This formulation generalizes the classical notion of connections on smooth manifolds—where commutative algebras of functions and vector bundles are central—by replacing them with noncommutative algebras and their (typically one-sided) modules. As Brzezi{\'n}ski remarked, \textquotedblleft this captures very well the classical context in which connections appear and brings it successfully to the realm of noncommutative geometry\textquotedblright\ \cite[p. 557]{Brzezinski2008}.

Brzezi{\'n}ski also pointed out that, from an algebraic standpoint, the above definition reveals only part of a broader picture. First, since noncommutative connections are built using the tensor functor, and given that this functor admits a right adjoint—the {\em hom-functor}—it is natural to explore whether one can define connection-like structures via this adjoint. Second, the dual of $M$ as a vector space becomes a right $A$-module, and a left connection (as defined above) does not induce a corresponding right connection on $M^*$. Considering the adjunction between the tensor and hom functors, any such induced structure would necessarily involve the hom-functor.

Brzezi{\'n}ski \cite{Brzezinski2008} introduced a class of connection-like maps defined on spaces of module homomorphisms, which he called {\em hom-connections} (also referred to as {\em divergences}, since when $A$ is an algebra of functions on $\mathbb{R}^n$ and $\Omega^1(A)$ is the standard module of 1-forms, the classical divergence operator is recovered \cite[p. 892]{Brzezinski2011}). Notably, he proved that hom-connections naturally emerge from (strong) connections on {\em noncommutative principal bundles}, and that any left connection on a bimodule (in the sense of Cuntz and Quillen \cite{CuntzQuillen1995}) gives rise to a hom-connection.

Brzezi{\'n}ski further analyzed how hom-connections can be induced through differentiable bimodules and morphisms of differential graded algebras. He showed that these structures extend to {\em higher-order forms}, and defined a notion of {\em curvature} for hom-connections. Repeated application of a hom-connection yields an operator governed by this curvature, resulting in a chain complex. When the curvature vanishes (i.e., the hom-connection is {\em flat}), this leads to a homology theory that may be viewed as the dual of the cochain complex associated to ordinary connections, central in the theory of noncommutative differential fibrations \cite{BeggsBrzezinski2005}.

Two years later, Brzezi{\'n}ski et al. \cite{BrzezinskiElKaoutitLomp2010} introduced a method for constructing a {\em differential calculi} that accommodate hom-connections. This method employs {\em twisted multi-derivations}, yielding first-order calculi $\Omega^1(A)$ that is free as both left and right $A$-module. From the geometrical point of view, $\Omega^1(A)$ models the module of sections of the cotangent bundle over a manifold described by $A$, and thus their construction corresponds to parallelizable manifolds or algebras of functions on local charts. Subsequently, Brzezi{\'n}ski observed that \textquotedblleft one should expect $\Omega^1(A)$ to be a finitely generated and projective module over $A$ (thus corresponding to sections of a non-trivial vector bundle by the Serre–Swan theorem)\textquotedblright\ \cite[p. 885]{Brzezinski2011}. In the same work, he generalized the construction in \cite{BrzezinskiElKaoutitLomp2010} to include finitely generated and projective modules.

In connection with differential calculi, there is also the concept of {\em smoothness of algebras}. According to Brzezi\'nski and Lomp \cite[Section 1]{BrzezinskiLomp2018}, the origins of this notion can be traced back to Grothendieck’s EGA \cite{Grothendieck1964}, where the idea of {\em formally smooth} commutative (or topological) algebras was introduced. Schelter \cite{Schelter1986} extended this idea to noncommutative algebras by defining formal smoothness in terms of the projectivity of the kernel of the multiplication map as a bimodule. This refinement replaced earlier, overly broad notions based solely on global dimension. In this vein, Cuntz and Quillen \cite{CuntzQuillen1995} referred to such algebras as {\em quasi-free}. Stafford and Zhang \cite{StaffordZhang1994} provided a more precise criterion: a Noetherian algebra is {\em smooth} if it has finite global dimension equal to the homological dimension of all its simple modules. Van den Bergh \cite{VandenBergh1998}, from a homological perspective, termed an algebra {\em homologically smooth} when it admits a finite projective bimodule resolution. This type of smoothness has been examined in the context of noncommutative spaces such as the noncommutative pillow, quantum teardrops, and quantum homogeneous spaces by Brzezi{\'n}ski \cite{Brzezinski2008, Brzezinski2014} and Kr\"ahmer \cite{Krahmer2012}.

An alternative notion, called {\em differential smoothness}, was later introduced by Brzezi{\'n}ski and Sitarz \cite{BrzezinskiSitarz2017}. This definition is based on differential graded algebras of fixed dimension that support a noncommutative analogue of the Hodge star isomorphism. It relies on the existence of a top-degree form and a version of Poincaré duality realized as an isomorphism between complexes of differential and integral forms. In contrast to homological smoothness, this notion is constructive and geometric. As they explained, \textquotedblleft the idea behind the {\em differential smoothness} of algebras is rooted in the observation that a classical smooth orientable manifold, in addition to the de Rham complex of differential forms, admits also the complex of {\em integral forms} isomorphic to the de Rham complex \cite[Section 4.5]{Manin1997}. The de Rham differential can be understood as a special left connection, while the boundary operator in the complex of integral forms is an example of a {\em right connection}\textquotedblright\ \cite[p. 413]{BrzezinskiSitarz2017}.

Brzezi{\'n}ski and various authors (e.g., \cite{Brzezinski2015, Brzezinski2016, BrzezinskiElKaoutitLomp2010, BrzezinskiLomp2018, BrzezinskiSitarz2017, DuboisViolette1988, DuboisVioletteKernerMadore1990, Karacuha2015, KaracuhaLomp2014, ReyesSarmiento2022}) have contributed to the study of differential smoothness for numerous algebras, such as the quantum two - and three-spheres, the quantum plane and disc, the noncommutative torus, the coordinate algebras of the quantum group $SU_q(2)$, the noncommutative pillow algebra, quantum cones, quantum polynomial algebras, certain Hopf algebra domains of Gelfand–Kirillov dimension two, various families of Ore extensions, 3-dimensional skew polynomial algebras, diffusion algebras, and deformations of classical orbifold coordinate algebras (including pillow orbifolds, singular cones, and lens spaces). Interestingly, several of these algebras are also homologically smooth in Van den Bergh’s sense.

Recently, the authors provided a characterization of the differential smoothness of different noncommutative algebras such as bi-quadratic algebras in three variables \cite{RubianoReyes2024DSBiquadraticAlgebras}, double Ore extensions \cite{RubianoReyes2024DSDoubleOreExtensions}, diffusion algebras \cite{RubianoReyes2024DSDifA}, and skew PBW extensions \cite{RubianoReyes2024DSSPBWKt, RubianoReyes2024DSSBWR}. The purpose of this paper is to present a general characterization for biquadratic algebras in $n$ variables with PBW basis introduced by Bavula \cite{Bavula2023}, and hence to extend the results formulated in \cite[Theorems 3.2 and 3.3]{RubianoReyes2024DSBiquadraticAlgebras}. Our results can be applied to several kinds of noncommutative polynomial type algebras studied by some authors (Bell et al. \cite{BellGoodearl1988, BellSmith1990}, Bueso et al. \cite{BuesoTorrecillasVerschoren2003}, Hinchcliffe \cite{Hinchcliffe2005}, Levandovskyy \cite{Levandovskyy2005}, Lezama et al. \cite{Fajardoetal2020, GallegoLezama2010, LezamaReyes2014, ReyesSuarez2017}, McConnell and Robson \cite{McConnellRobson2001}, Pyatov and coauthors \cite{IsaevPyatovRittenberg2001, PyatovTwarock2002}, Redman \cite{RedmanPhDThesis1996, Redman1999}, Rosenberg \cite{Rosenberg1995}, and Seiler \cite{SeilerBook2010}).

The structure of the paper is as follows. Section \ref{PreliminariesDifferentialsmoothnessofbi-quadraticalgebras} introduces the foundational concepts concerning differential smoothness and the class of bi-quadratic algebras equipped with PBW basis. In Section \ref{Differentialandintegralcalculusbi-quadraticalgebras}, we present the main original results of this work. We generalize Brzezi{\'n}ski's ideas \cite{Brzezinski2015} developed for skew polynomial rings over the commutative polynomial algebra $\Bbbk[t]$ to the context of bi-quadratic algebras with PBW basis. Theorem \ref{Firsttheoremsmoothnessbi-quadraticalgebras} provides sufficient conditions under which such algebras are differentially smooth, while Theorem \ref{Secondtheoremsmoothnessbi-quadraticalgebras} offers sufficient conditions that preclude differential smoothness. Section \ref{Examplesdifferentiallysmoothngenerators} contains examples that illustrate both results. In particular, in Remark \ref{CQWABavula2024} we see that the $3$-{\em cyclic quantum Weyl algebras} $A(\alpha, \beta, \gamma)$ investigated by Bavula \cite{Bavula2024} are differentially smooth.

Throughout the paper, $\mathbb{N}$ denotes the set of natural numbers, including zero. The word ring means an associative ring with identity not necessarily commutative. $Z(R)$ denotes the center of the ring $R$. All vector spaces and algebras (always associative and with unit) are over a fixed field $\Bbbk$. $\Bbbk^{*}$ denotes the non-zero elements of $\Bbbk$. As usual, the symbols $\mathbb{R}$ and $\mathbb{C}$ denote the fields of real and complex numbers, respectively. 

\section{Preliminaries}\label{PreliminariesDifferentialsmoothnessofbi-quadraticalgebras}

We start by recalling the preliminaries on differential smoothness of algebras and bi-quadratic algebras with PBW basis.

\subsection{Differential smoothness of algebras}\label{DefinitionsandpreliminariesDSA}





\begin{definition}[{\cite[Section 2.1]{BrzezinskiSitarz2017}}]
\begin{enumerate}
    \item [\rm (i)] A {\em differential graded algebra} is a non-negatively graded algebra $\Omega$ with the product denoted by $\wedge$ together with a degree-one linear map $$
    d:\Omega^{\bullet} \to \Omega^{\bullet +1}
    $$ 
    
    that satisfies the graded Leibniz's rule and is such that $d \circ d = 0$. 
    
    \item [\rm (ii)] A differential graded algebra $(\Omega, d)$ is a {\em calculus over an algebra} $A$ if $\Omega^0 A = A$ and 
    $$
    \Omega^n A = A\ dA \wedge dA \wedge \dotsb \wedge dA
    $$ 
    
    ($dA$ appears $n$-times) for all $n\in \mathbb{N}$. This last is called the {\em density condition}. We write $(\Omega A, d)$ with 
    $$
    \Omega A = \bigoplus_{n\in \mathbb{N}} \Omega^{n}A. 
    $$
    
    By using the Leibniz's rule, it follows that 
    $$
    \Omega^n A = dA \wedge dA \wedge \dotsb \wedge dA\ A. 
    $$
    
    A differential calculus $\Omega A$ is said to be {\em connected} if ${\rm ker}(d\mid_{\Omega^0 A}) = \Bbbk$.
    
    \item [\rm (iii)] A calculus $(\Omega A, d)$ is said to have {\em dimension} $n$ if $\Omega^n A\neq 0$ and $\Omega^m A = 0$ for all $m > n$. An $n$-dimensional calculus $\Omega A$ {\em admits a volume form} if $\Omega^n A$ is isomorphic to $A$ as a left and right $A$-module. 
\end{enumerate}
\end{definition}

The existence of a right $A$-module isomorphism means that there is a free generator, say $\omega$, of $\Omega^n A$ (as a right $A$-module), i.e. $\omega \in \Omega^n A$, such that all elements of $\Omega^n A$ can be uniquely expressed as $\omega a$ with $a \in A$. If $\omega$ is also a free generator of $\Omega^n A$ as a left $A$-module, this is said to be a {\em volume form} on $\Omega A$.

The right $A$-module isomorphism $\Omega^n A \to A$ corresponding to a volume form $\omega$ is denoted by $\pi_{\omega}$, i.e.
\begin{equation}\label{BrzezinskiSitarz2017(2.1)}
\pi_{\omega} (\omega a) = a, \quad {\rm for\ all}\ a\in A.
\end{equation}

By using that $\Omega^n A$ is also isomorphic to $A$ as a left $A$-module, any free generator $\omega $ induces an algebra endomorphism $\nu_{\omega}$ of $A$ by the formula
\begin{equation}\label{BrzezinskiSitarz2017(2.2)}
    a \omega = \omega \nu_{\omega} (a).
\end{equation}

Note that if $\omega$ is a volume form, then $\nu_{\omega}$ is an algebra automorphism.

We proceed to recall the key ingredients of the {\em integral calculus} on $A$ as dual to its differential calculus. For more details, see Brzezinski et al. \cite{Brzezinski2008, BrzezinskiElKaoutitLomp2010}.

Let $(\Omega A, d)$ be a differential calculus on $A$. The space of $n$-forms $\Omega^n A$ is an $A$-bimodule. Consider $\mathcal{I}_{n}A$ the right dual of $\Omega^{n}A$, the space of all right $A$-linear maps $\Omega^{n}A\rightarrow A$, that is, 
$$
\mathcal{I}_{n}A := {\rm Hom}_{A}(\Omega^{n}(A),A).
$$

Notice that each of the $\mathcal{I}_{n}A$ is an $A$-bimodule with the actions
\begin{align*}
    (a\cdot\phi\cdot b)(\omega)=a\phi(b\omega),\quad {\rm for\ all}\ \phi \in \mathcal{I}_{n}A,\ \omega \in \Omega^{n}A\ {\rm and}\ a,b \in A.
\end{align*}

The direct sum of all the $\mathcal{I}_{n}A$, that is, $\mathcal{I}A = \bigoplus\limits_{n} \mathcal{I}_n A$, is a right $\Omega A$-module with action given by
\begin{align}\label{BrzezinskiSitarz2017(2.3)}
    (\phi\cdot\omega)(\omega')=\phi(\omega\wedge\omega'),\quad {\rm for\ all}\ \phi\in\mathcal{I}_{n + m}A, \ \omega\in \Omega^{n}A \ {\rm and} \ \omega' \in \Omega^{m}A.
\end{align}

\begin{definition}[{\cite[Definition 2.1]{Brzezinski2008}}]
A {\em divergence} (also called {\em hom-connection}) on $A$ is a linear map $\nabla: \mathcal{I}_1 A \to A$ such that
\begin{equation}\label{BrzezinskiSitarz2017(2.4)}
    \nabla(\phi \cdot a) = \nabla(\phi) a + \phi(da), \quad {\rm for\ all}\ \phi \in \mathcal{I}_1 A \ {\rm and} \ a \in A.
\end{equation}  
\end{definition}

Note that a divergence can be extended to the whole of $\mathcal{I}A$, 
\[
\nabla_n: \mathcal{I}_{n+1} A \to \mathcal{I}_{n} A,
\]

by considering
\begin{equation}\label{BrzezinskiSitarz2017(2.5)}
\nabla_n(\phi)(\omega) = \nabla(\phi \cdot \omega) + (-1)^{n+1} \phi(d \omega), \quad {\rm for\ all}\ \phi \in \mathcal{I}_{n+1}(A)\ {\rm and} \ \omega \in \Omega^n A.
\end{equation}

By putting together (\ref{BrzezinskiSitarz2017(2.4)}) and (\ref{BrzezinskiSitarz2017(2.5)}), we get the Leibniz's rule 
\begin{equation}
    \nabla_n(\phi \cdot \omega) = \nabla_{m + n}(\phi) \cdot \omega + (-1)^{m + n} \phi \cdot d\omega,
\end{equation}

for all elements $\phi \in \mathcal{I}_{m + n + 1} A$ and $\omega \in \Omega^m A$ \cite[Lemma 3.2]{Brzezinski2008}. In the case $n = 0$, if ${\rm Hom}_A(A, M)$ is canonically identified with $M$, then $\nabla_0$ reduces to the classical Leibniz's rule.

\begin{definition}[{\cite[Definition 3.4]{Brzezinski2008}}]
The right $A$-module map 
\[
F = \nabla_0 \circ \nabla_1: {\rm Hom}_A(\Omega^{2} A, M) \to M
\] 

is called a {\em curvature} of a hom-connection $(M, \nabla_0)$. $(M, \nabla_0)$ is said to be {\em flat} if its curvature is the zero map, that is, if $\nabla \circ \nabla_1 = 0$. This condition implies that $\nabla_n \circ \nabla_{n+1} = 0$ for all $n\in \mathbb{N}$.
\end{definition}

$\mathcal{I} A$ together with the $\nabla_n$ form a chain complex called the {\em complex of integral forms} over $A$. The cokernel map of $\nabla$, that is, $\Lambda: A \to {\rm Coker} \nabla = A / {\rm Im} \nabla$ is said to be the {\em integral on $A$ associated to} $\mathcal{I}A$.

Given a left $A$-module $X$ with action $a\cdot x$, for all $a\in A,\ x \in X$, and an algebra automorphism $\nu$ of $A$, the notation $^{\nu}X$ stands for $X$ with the $A$-module structure twisted by $\nu$, i.e. with the $A$-action $a\otimes x \mapsto \nu(a)\cdot x $.

The following definition of an \textit{integrable differential calculus} seeks to portray a version of Hodge star isomorphisms between the complex of differential forms of a differentiable manifold and a complex of dual modules of it \cite[p. 112]{Brzezinski2015}. 

\begin{definition}[{\cite[Definition 2.1]{BrzezinskiSitarz2017}}]
An $n$-dimensional differential calculus $(\Omega A, d)$ is said to be {\em integrable} if $(\Omega A, d)$ admits a complex of integral forms $(\mathcal{I}A, \nabla)$ for which there exist an algebra automorphism $\nu$ of $A$ and $A$-bimodule isomorphisms \linebreak $\Theta_k: \Omega^{k} A \to ^{\nu} \mathcal{I}_{n-k}A$, $k = 0, \dotsc, n$, rendering commmutative the following diagram:
\[
\begin{tikzcd}
A \arrow{r}{d} \arrow{d}{\Theta_0} & \Omega^{1} A \arrow{d}{\Theta_1} \arrow{r}{d} & \Omega^2 A  \arrow{d}{\Theta_2} \arrow{r}{d} & \dotsb \arrow{r}{d} & \Omega^{n-1} A \arrow{d}{\Theta_{n-1}} \arrow{r}{d} & \Omega^n A  \arrow{d}{\Theta_n} \\ ^{\nu} \mathcal{I}_n A \arrow[swap]{r}{\nabla_{n-1}} & ^{\nu} \mathcal{I}_{n-1} A \arrow[swap]{r}{\nabla_{n-2}} & ^{\nu} \mathcal{I}_{n-2} A \arrow[swap]{r}{\nabla_{n-3}} & \dotsb \arrow[swap]{r}{\nabla_{1}} & ^{\nu} \mathcal{I}_{1} A \arrow[swap]{r}{\nabla} & ^{\nu} A
\end{tikzcd}
\]

The $n$-form $\omega:= \Theta_n^{-1}(1)\in \Omega^n A$ is called an {\em integrating volume form}. 
\end{definition}

The algebra of complex matrices $M_n(\mathbb{C})$ with the $n$-dimensional calculus generated by derivations presented by Dubois-Violette et al. \cite{DuboisViolette1988, DuboisVioletteKernerMadore1990}, the quantum group $SU_q(2)$ with the three-dimensional left covariant calculus developed by Woronowicz \cite{Woronowicz1987} and the quantum standard sphere with the restriction of the above calculus, are examples of algebras admitting integrable calculi. For more details on the subject, see Brzezi\'nski et al. \cite{BrzezinskiElKaoutitLomp2010}. 

The following proposition shows that the integrability of a differential calculus can be defined without explicit reference to integral forms. This allows us to guarantee the integrability by considering the existence of finitely generator elements that allow to determine left and right components of any homogeneous element of $\Omega(A)$.

\begin{proposition}[{\cite[Theorem 2.2]{BrzezinskiSitarz2017}}]\label{integrableequiva} 
Let $(\Omega A, d)$ be an $n$-dimensional differential calculus over an algebra $A$. The following assertions are equivalent:
\begin{enumerate}
    \item [\rm (1)] $(\Omega A, d)$ is an integrable differential calculus.
    
    \item [\rm (2)] There exists an algebra automorphism $\nu$ of $A$ and $A$-bimodule isomorphisms 
    $$
    \Theta_k : \Omega^k A \rightarrow \ ^{\nu}\mathcal{I}_{n-k}A, \quad k = 0, \ldots, n, 
    $$ 
    
    such that, for all $\omega'\in \Omega^k A$ and $\omega''\in \Omega^mA$,
    \begin{align*}
        \Theta_{k+m}(\omega'\wedge\omega'')=(-1)^{(n-1)m}\Theta_k(\omega')\cdot\omega''.
    \end{align*}
    
    \item [\rm (3)] There exists an algebra automorphism $\nu$ of $A$ and an $A$-bimodule map $\vartheta:\Omega^nA\rightarrow\ ^{\nu}A$ such that all left multiplication maps
    \begin{align*}
    \ell_{\vartheta}^{k}:\Omega^k A &\ \rightarrow \mathcal{I}_{n-k}A, \\
    \omega' &\ \mapsto \vartheta\cdot\omega', \quad k = 0, 1, \dotsc, n,
    \end{align*}
    where the actions $\cdot$ are defined by {\rm (}\ref{BrzezinskiSitarz2017(2.3)}{\rm )}, are bijective.
    
    \item [\rm (4)] $(\Omega A, d)$ has a volume form $\omega$ such that all left multiplication maps
    \begin{align*}
        \ell_{\pi_{\omega}}^{k}:\Omega^k A &\ \rightarrow \mathcal{I}_{n-k}A, \\
        \omega' &\ \mapsto \pi_{\omega} \cdot \omega', \quad k=0,1, \dotsc, n-1,
    \end{align*}
    
    where $\pi_{\omega}$ is defined by {\rm (}\ref{BrzezinskiSitarz2017(2.1)}{\rm )}, are bijective.
\end{enumerate}
\end{proposition}

A volume form $\omega\in \Omega^nA$ is an {\em integrating form} if and only if it satisfies condition $(4)$ of Proposition \ref{integrableequiva} \cite[Remark 2.3]{BrzezinskiSitarz2017}.

The most interesting cases of differential calculi are those where $\Omega^k A$ are finitely generated and projective right or left (or both) $A$-modules \cite{Brzezinski2011}.

\begin{proposition}\label{BrzezinskiSitarz2017Lemmas2.6and2.7}
\begin{enumerate}
\item [\rm (1)] \cite[Lemma 2.6]{BrzezinskiSitarz2017} Consider $(\Omega A, d)$ an integrable and $n$-dimensional calculus over $A$ with integrating form $\omega$. Then $\Omega^{k} A$ is a finitely generated projective right $A$-module if there exist a finite number of forms $\omega_i \in \Omega^{k} A$ and $\overline{\omega}_i \in \Omega^{n-k} A$ such that, for all $\omega' \in \Omega^{k} A$, we have that 
\begin{equation*}
\omega' = \sum_{i} \omega_i \pi_{\omega} (\overline{\omega}_i \wedge \omega').
\end{equation*}

\item [\rm (2)] \cite[Lemma 2.7]{BrzezinskiSitarz2017} Let $(\Omega A, d)$ be an $n$-dimensional calculus over $A$ admitting a volume form $\omega$. Assume that for all $k = 1, \ldots, n-1$, there exists a finite number of forms $\omega_{i}^{k},\overline{\omega}_{i}^{k} \in \Omega^{k}(A)$ such that for all $\omega'\in \Omega^kA$, we have that
\begin{equation*}
\omega'=\displaystyle\sum_i\omega_{i}^{k}\pi_\omega(\overline{\omega}_{i}^{n-k}\wedge\omega')=\displaystyle\sum_i\nu_{\omega}^{-1}(\pi_\omega(\omega'\wedge\omega_{i}^{n-k}))\overline{\omega}_{i}^{k},
\end{equation*}

where $\pi_{\omega}$ and $\nu_{\omega}$ are defined by {\rm (}\ref{BrzezinskiSitarz2017(2.1)}{\rm )} and {\rm (}\ref{BrzezinskiSitarz2017(2.2)}{\rm )}, respectively. Then $\omega$ is an integral form and all the $\Omega^{k}A$ are finitely generated and projective as left and right $A$-modules.
\end{enumerate}
\end{proposition}

Brzezi\'nski and Sitarz \cite[p. 421]{BrzezinskiSitarz2017} asserted that to connect the integrability of the differential graded algebra $(\Omega A, d)$ with the algebra $A$, it is necessary to relate the dimension of the differential calculus $\Omega A$ with that of $A$, and since we are dealing with algebras that are deformations of coordinate algebras of affine varieties, the {\em Gelfand-Kirillov dimension} introduced by Gelfand and Kirillov \cite{GelfandKirillov1966, GelfandKirillov1966b} seems to be the best suited. Briefly, given an affine $\Bbbk$-algebra $A$, the {\em Gelfand-Kirillov dimension of} $A$, denoted by ${\rm GKdim}(A)$, is given by
\[
{\rm GKdim}(A) := \underset{n\to \infty}{\rm lim\ sup} \frac{{\rm log}({\rm dim}\ V^{n})}{{\rm log}\ n},
\]

where $V$ denotes a finite-dimensional subspace of $A$ that generates $A$ as an algebra. This definition does not depend on the particular choice of $V$. If $A$ is not affine, its Gelfand-Kirillov dimension is defined as the supremum of the Gelfand-Kirillov dimensions of all affine subalgebras of $A$. An affine domain with Gelfand-Kirillov dimension zero is exactly a division ring that is finite-dimensional over its center. In the affine case with Gelfand-Kirillov dimension one over $\Bbbk$, the algebra is a finite module over its center, and hence satisfies a polynomial identity. In a broad sense, this dimension quantifies how far the algebra $A$ is from being finite-dimensional. For further details on this notion, we refer the reader to the comprehensive exposition by Krause and Lenagan \cite{KrauseLenagan2000}.

We arrive to the key notion of this paper.

\begin{definition}[{\cite[Definition 2.4]{BrzezinskiSitarz2017}}]\label{BrzezinskiSitarz2017Definition2.4}
An affine algebra $A$ with integer Gelfand-Kirillov dimension $n$ is said to be {\em differentially smooth} if it admits an $n$-dimensional connected integrable differential calculus $(\Omega A, d)$.
\end{definition}

From Definition \ref{BrzezinskiSitarz2017Definition2.4} it follows that a differentially smooth algebra comes equipped with a well-behaved differential structure and with the precise concept of integration \cite[p. 2414]{BrzezinskiLomp2018}.

\begin{example}\label{examplesDSalgebrasBre}
As we said in the Introduction, several examples of noncommutative algebras have been proved to be differentially smooth (e.g. \cite{Brzezinski2015, BrzezinskiElKaoutitLomp2010, BrzezinskiLomp2018, BrzezinskiSitarz2017, Karacuha2015, KaracuhaLomp2014, ReyesSarmiento2022}). For instance, the polynomial algebra $\Bbbk[x_1, \dotsc, x_n]$ has the Gelfand-Kirillov dimension $n$ and the usual exterior algebra is an $n$-dimensional integrable calculus, whence $\Bbbk[x_1, \dotsc, x_n]$ is differentially smooth. From \cite{BrzezinskiElKaoutitLomp2010} we know that the coordinate algebras of the quantum group $SU_q(2)$, the standard quantum Podle\'s and the quantum Manin plane are differentially smooth. Of course, there are examples of algebras that are not differentially smooth. Consider the commutative algebra $A = \mathbb{C}[x, y] / \langle xy \rangle$. A proof by contradiction shows that for this algebra there are no one-dimensional connected integrable calculi over $A$, so it cannot be differentially smooth \cite[Example 2.5]{BrzezinskiSitarz2017}.
\end{example}

\subsection{Bi-quadratic algebras with PBW basis}\label{BiquadraticalgebrasPBWbasis}

For a natural number $n\ge 2$, a family $M = (m_{ij})_{i > j}$ of elements $m_{ij}$ belonging to $R$ ($1\le j < i \le n$) is called a {\em lower triangular half-matrix} with coefficients in $R$. The set of all such matrices is denoted by $L_n(R)$.

\begin{definition}[{\cite[Section 1]{Bavula2023}}]
If $\sigma = (\sigma_1, \dotsc, \sigma_n)$ is an $n$-tuple of commuting endomorphisms of $R$, $\delta = (\delta_1, \dotsc, \delta_n)$ is an $n$-tuple of $\sigma$-endomorphisms of $R$ (that is, $\delta_i$ is a $\sigma_i$-derivation of $R$ for $i=1,\dotsc, n$), $Q = (q_{ij})\in L_n(Z(R))$, $\mathbb{A}:= (a_{ij, k})$ where $a_{ij, k}\in R$, $1\le j < i \le n$ and $k = 1,\dotsc, n$, and $\mathbb{B}:= (b_{ij})\in L_n(R)$, the {\em skew bi-quadratic algebra} ({\em SBQA}) $A = R[x_1,\dotsc, x_n;\sigma, \delta, Q, \mathbb{A}, \mathbb{B}]$ is a ring generated by the ring $R$ and elements $x_1, \dotsc, x_n$ subject to the defining relations
\begin{align}
    x_ir = &\ \sigma_i(r)x_i + \delta_i(r),\quad {\rm for}\ i = 1, \dotsc, n,\ {\rm and\ every}\ r\in R, \label{Bavula2023(1)} \\
    x_ix_j - q_{ij}x_jx_i = &\ \sum_{k=1}^{n} a_{ij, k}x_k + b_{ij},\quad {\rm for\ all}\ j < i.\label{Bavula2023(2)}
\end{align}

If $\sigma_i = {\rm id}_R$ and $\delta_i = 0$ for $i = 1,\dotsc, n$, the ring $A$ is called the {\em bi-quadratic algebra} ({\em BQA}) and is denoted by $A = R[x_1, \dotsc, x_n; Q, \mathbb{A}, \mathbb{B}]$. $A$ has {\em PBW basis} if 
$$
A = \bigoplus\limits_{\alpha \in \mathbb{N}^{n}} Rx^{\alpha}
$$ 

where $x^{\alpha} = x_1^{\alpha_1}\dotsb x_n^{\alpha_n}$.
\end{definition}

On the set $W_n$ of all words in the alphabet $\{x_1, \dotsc, x_n\}$, Bavula considered the {\em degree-by-lexicographic ordering} where $x_1 < \dotsb < x_n$. More exactly, $x_{i_1} \dotsb x_{i_s} < x_{j_1} \dotsb x_{j_t}$ if either $s < t$ or $s = t, \ i_1 = j_1, \dotsc, i_k = j_k$ and $i_{k+1} < j_{k+1}$ for some $k$ such that $1 \le k < s$. Hence, if $A = R[x_1, \dotsc, x_n; Q, \mathbb{A}, \mathbb{B}]$ is a quadratic algebra where $n\ge 3$, then for each triple $i, j, k \in \{1, \dotsc, n\}$ such that $ i < j < k$, there are exactly two different ways to simplify the product $x_kx_jx_i$ with respect to this order given by
\begin{align*}
    x_k x_j x_i = &\ q_{kj} q_{ki} q_{ji} x_i x_j x_k + \sum_{|\alpha| \le 2} c_{k, j, i, \alpha} x^{\alpha}, \\
    x_k x_j x_i = &\ q_{kj} q_{ki} q_{ji} x_i x_j x_k + \sum_{|\alpha| \le 2} c'_{k, j, i, \alpha} x^{\alpha},
\end{align*}

where in the first (resp. second) equality we start to simplify the product with the relation $x_k x_j = q_{kj} x_j x_k + \dotsb$ (resp. $x_j x_i = q_{ji} x_i x_j + \dotsb$) \cite[p. 696]{Bavula2023}. Thus, the defining relations (\ref{Bavula2023(2)}) are consistent (i.e. $A\neq 0$) and $A = \bigoplus\limits_{\alpha \in \mathbb{N}^n} Rx^{\alpha}$ if and only if for all triples $i, j, k \in \{1, \dotsc, n\}$ such that $i < j < k$, we have that $ c_{k, j, i, \alpha} = c'_{k, j, i, \alpha}$. If this is the case, then for all $\sigma \in S_n$ ($S_n$ denotes the {\em symmetric group of order $n$}) we have that $A = \bigoplus\limits_{\alpha \in \mathbb{N}^n} Rx_{\sigma}^{\alpha}$ where $x_{\sigma}^{\alpha} = x_{\sigma(1)}^{\alpha_1} \dotsb x_{\sigma(n)}^{\alpha_n}$ and $\alpha = (\alpha_1, \dotsc, \alpha_n)$ \cite[Theorem 1.1]{Bavula2023}.

\begin{proposition}\label{BiquadraticalgebrasPBWbasisTwogenerators}
The classification of bi-quadratic algebras on two generators over a field $\Bbbk$ is as follows: if $q\in \Bbbk^{*}$ and $a, b, c \in \Bbbk$, then the algebra
\[
A = \Bbbk[x_1, x_2; q, a, b, c] := \Bbbk\{ x_1, x_2\} / \langle x_2 x_1 - qx_1x_2 - ax_1 - bx_2 - c\rangle
\]

is a bi-quadratic algebra on two generators \cite[p. 704]{Bavula2023}. The algebra $A = \Bbbk[x_1][x_2; \sigma, \delta]$ is a {\em skew polynomial ring} or {\em Ore extension} {\rm (}introduced by Ore \cite{Ore1933}{\rm )} of $\Bbbk[x_1]$ where $\sigma(x_1) = qx_1 + b$ and $\delta(x_1) = ax_1 + c$. In this way, the algebra $A$ is a Noetherian domain with PBW basis, $A = \bigoplus\limits_{i, j \in \mathbb{N}} \Bbbk x_1^{i} x_2^{j}$ and the scalars $a, b, c$ are arbitrary.

Up to isomorphism, there are only five bi-quadratic algebras of $\Bbbk$ on two generators \cite[Theorem 2.1]{Bavula2023}: 
    \begin{enumerate}
        \item [\rm (1)] The {\em polynomial algebra} $\Bbbk[x_1, x_2]$;
        
        \item [\rm (2)] The {\em Weyl algebra} $A_1(\Bbbk) = \Bbbk\{x_1, x_2\} / \langle x_1x_2 - x_2x_1 - 1\rangle$;
        
        \item [\rm (3)] The {\em universal enveloping algebra of the Lie algebra} $\mathfrak{n}_2 = \langle x_1, x_2\mid [x_2, x_1] = x_1\rangle$, that is, $U(\mathfrak{n}_2) = \Bbbk\{x_1, x_2\} / \langle x_2x_1 - x_1x_2 - x_1\rangle$;
        
        \item [\rm (4)] The {\em quantum plane} ({\em Manin's plane}) $\mathcal{O}_q(\Bbbk) = \Bbbk \{x_1, x_2\} / \langle x_2 x_1 - qx_1 x_2\rangle$, where $q\in \Bbbk\ \backslash\ \{0,1\}$;
        
        \item [\rm (5)] The {\em quantum Weyl algebra} $A_{1}^{q}(\Bbbk) = \Bbbk \{x_1, x_2\} / \langle x_2x_1 - qx_1x_2 - 1\rangle$, where $q\in \Bbbk\ \backslash\ \{0,1\}$.
    \end{enumerate}
\end{proposition}

Proposition \ref{Bavula2023Theorem1.3} gives necessary and sufficient conditions for the bi-quadratic algebras on three-generators have a PBW basis (c.f. \cite{BellSmith1990, ReyesSuarez20173D, ReyesSarmiento2022, Rosenberg1995}). As we will see in Section \ref{DSBiquadraticalgebrasThreegenerators}, this is very useful to characterize the differential smoothness of these algebras.

\begin{proposition}[{\cite[Theorem 1.3]{Bavula2023}}]\label{Bavula2023Theorem1.3}
If $A = \Bbbk[x_1, x_2, x_3;Q, \mathbb{A}, \mathbb{B}]$ is a bi-quadratic algebra where $Q = (q_1, q_2, q_3) \in \Bbbk^{* 3}$,
\[
\mathbb{A} = \begin{bmatrix}
    a & b & c \\ \alpha & \beta & \gamma \\ \lambda & \mu & \nu
\end{bmatrix}\quad {\rm and}\quad \mathbb{B} = \begin{bmatrix} b_1 \\ b_2 \\ b_3 \end{bmatrix}.
\]

The algebra $A$ is generated over $\Bbbk$ by the elements $x_1, x_2$ and $x_3$ subject to the defining relations
\begin{align}
    x_2 x_1 - q_1x_1x_2 = &\ ax_1 + bx_2 + cx_3 + b_1, \label{Bavula2023(8)} \\
    x_3 x_1 - q_2 x_1 x_3 = &\ \alpha x_1 + \beta x_2 + \gamma x_3 + b_2, \label{Bavula2023(9)} \\
    x_3 x_2 - q_3x_2 x_3 = &\ \lambda x_1 + \mu x_2 + \nu x_3 + b_3. \label{Bavula2023(10)}
\end{align}

The relations {\rm(}\ref{Bavula2023(8)}{\rm )}, {\rm (}\ref{Bavula2023(9)}{\rm )} and {\rm (}\ref{Bavula2023(10)}{\rm )} are consistent and $A = \bigoplus\limits_{\alpha \in \mathbb{N}^3} \Bbbk x^{\alpha}$ where $x^{\alpha} = x_1^{\alpha_1} x_2^{\alpha_2}x_3^{\alpha_3}$ if and only if the following conditions hold:
\begin{align}
&\ (1 - q_3)\alpha =  (1 - q_2)\mu, \label{Bavula2023(11)} \\
&\ (1 - q_3) a = (1 - q_1)\nu, \label{Bavula2023(12)} \\
&\ (1 - q_2) b = (1 - q_1)\gamma \label{Bavula2023(13)} \\
&\ (1 - q_1 q_2) \lambda = 0, \label{Bavula2023(14)} \\
&\ (q_1 - q_3) \beta = 0, \label{Bavula2023(15)} \\
&\ (1 - q_2 q_3)c = 0, \label{Bavula2023(16)} \\
&\ ((1-q_3)\alpha -\mu)a+(b+q_1\gamma)\lambda-\nu\alpha+(q_1q_2-1)b_3=0, \label{Bavula2023(17)} \\
&\ (a-\nu)\beta + q_1\gamma\mu -q_3\alpha b+(q_1-q_3)b_2=0, \label{Bavula2023(18)}\\
&\ (a+(q_1-1)\nu)\gamma+b\nu-(\mu+q_3\alpha)c+(1-q_2q_3)b_1=0, \quad {\rm and} \label{Bavula2023(19)}\\
&\ -(\mu + q_3\alpha)b_1+(a-\nu)b_2+(b+q_1\gamma)b_3=0. \label{Bavula2023(20)}
\end{align}

Furthermore, if $A = \bigoplus\limits_{\alpha \in \mathbb{N}^3} \Bbbk x^{\alpha}$ where $x^{\alpha} = x_1^{\alpha_1} x_2^{\alpha_2} x_3^{\alpha_3}$ then $A = \bigoplus\limits_{\alpha \in \mathbb{N}^3} \Bbbk x^{\alpha}_{\sigma}$ for all $\sigma \in S_3$ where $x^{\alpha}_{\sigma} = x^{\alpha_1}_{\sigma(1)} x^{\alpha_2}_{\sigma(2)} x^{\alpha_3}_{\sigma(3)}$.
\end{proposition}

From now on, the expression \textquotedblleft bi-quadratic algebra\textquotedblright\ means \textquotedblleft bi-quadratic algebra with PBW basis\textquotedblright.

\begin{example}\label{Examplesbi-quadraticthreegenerators}
Let us see some examples of bi-quadratic algebras on three generators.
\begin{enumerate}
    \item [\rm (a)] The universal enveloping algebra of any 3-dimensional Lie algebra.
    
    \item [\rm (b)] The 3-{\em dimensional quantum space} $\mathbb{A}_{q_1, q_2, q_3}^{3} := \Bbbk[x_1, x_2, x_3; Q, \mathbb{A} = 0, \mathbb{B} = 0]$.
    
    \item [\rm (c)] Following Havli\v{c}ek et al. \cite[p. 79]{HavlicekKlimykPosta2000}, the $ \Bbbk$-algebra $U_q'(\mathfrak{so}_3)$ is generated by the indeterminates $I_1, I_2$, and $I_3$ subject to the relations given by
\begin{align*}
    I_2I_1 - qI_1I_2 = &\ -q^{\frac{1}{2}}I_3, \\
    I_3I_1 - q^{-1}I_1I_3 = &\ q^{-\frac{1}{2}}I_2, \quad {\rm and} \\
    I_3I_2 - qI_2I_3 = &\ -q^{\frac{1}{2}}I_1, \quad {\rm for}\ q\in \Bbbk\ \backslash \ \{0, \pm 1\}. 
\end{align*}

\item [\rm (d)] Zhedanov \cite[Section 1]{Zhedanov1991} introduced the {\em Askey-Wilson algebra} $AW(3)$ as the $\mathbb{R}$-algebra generated by three operators $K_0, K_1$, and $K_2$, that satisfy the commutation relations 
\begin{align*}
[K_0, K_1]_{\omega} = &\ K_2, \\
[K_2, K_0]_{\omega} = &\ BK_0 + C_1K_1 + D_1, \quad {\rm and} \\
[K_1, K_2]_{\omega} = &\ BK_1 + C_0K_0 + D_0, 
\end{align*}

where $B, C_0, C_1, D_0$, and $D_1$ are elements of $\mathbb{R}$ that represent the structure constants of the algebra, and the $q$-commutator $[ - , -]_{\omega}$ is given by $[\square, \triangle]_{\omega}:= e^{\omega}\square \triangle - e^{- \omega}\triangle \square$, where $\omega\in \mathbb{R}$. Notice that in the limit $\omega \to 0$, the algebra AW(3) becomes an ordinary Lie algebra with three generators ($D_0$ and $D_1$ are included among the structure constants of the algebra in order to take into account algebras of Heisenberg-Weyl type). The relations defining the algebra can be written as 
\begin{align*}
    e^{\omega}K_0K_1 - e^{-\omega}K_1K_0 = &\ K_2,\\
    e^{\omega} K_2K_0 - e^{-\omega}K_0 K_2 = &\ BK_0 + C_1K_1 + D_1, \quad {\rm and} \\
    e^{\omega}K_1K_2 - e^{-\omega}K_2K_1 = &\ BK_1 + C_0K_0 + D_0.
\end{align*}
\end{enumerate}
\end{example}

With the aim of classifying bi-quadratic algebras on three generators, Bavula \cite{Bavula2023, BavulaAlKhabyah2023} considered $Q = (q_1, q_2, q_3) \in \Bbbk^{* 3}$ in Proposition \ref{Bavula2023Theorem1.3} into the following four cases:
\begin{itemize}
    \item $q_1 = q_2 = q_3 = 1$ (Lie type); 
    \item $q_1 \neq 1,\ q_2 = q_3 = 1$; 
    \item $q_1\neq 1,\ q_2\neq 1,\ q_3 = 1$, and
    \item $q_1\neq 1,\ q_2\neq 1,\ q_3\neq 1$.
\end{itemize}

As it can be seen from Bavula's papers, there are exactly $44$ types (up to isomorphism and considering $\sqrt{\Bbbk}\subseteq \Bbbk$ and $\sqrt[3]{\Bbbk}\subseteq \Bbbk$) of non-isomorphic algebras of bi-quadratic algebras on three generators.

\section{Differential and integral calculus}\label{Differentialandintegralcalculusbi-quadraticalgebras}

In this section we investigate the differential smoothness of bi-quadratic algebras on $n$ generators with PBW basis. Before, we say a few words about the smoothness of the bi-quadratic algebras on two generators, three generators and four generators.

\subsection{Bi-quadratic algebras on two generators}\label{DSBiquadraticalgebrasTwogenerators}

Brzezi{\'n}ski \cite{Brzezinski2015} characterized the differential smoothness of skew polynomial rings of the form $\Bbbk[t][x; \sigma_{q, r}, \delta_{p(t)}]$ where $\sigma_{q, r}(t) = qt + r$, with $q, r \in \Bbbk,\ q\neq 0$, and the $\sigma_{q, r}-$derivation $\delta_{p(t)}$ is defined as
\[
\delta_{p(t)} (f(t)) = \frac{f(\sigma_{q, r}(t)) - f(t)}{\sigma_{q, r}(t) - t} p(t),
\]

for an element $p(t) \in \Bbbk[t]$. $\delta_{p(t)}(f(t))$ is a suitable limit when $q = 1$ and $r = 0$, that is, when $\sigma_{q, r}$ is the identity map of $\Bbbk[t]$.

For the maps
\begin{equation}\label{Brzezinski2015(3.4)}
\nu_t(t) = t,\quad \nu_t(x) = qx + p'(t)\quad {\rm and}\quad \nu_x(t) = \sigma_{q, r}^{-1}(t),\quad \nu_x(x) = x,
\end{equation}

where $p'(t)$ is the classical $t$-derivative of $p(t)$, Brzezi{\'n}ski \cite[Lemma 3.1]{Brzezinski2015} showed that all of them simultaneously extend to algebra automorphisms $\nu_t$ and $\nu_x$ of $\Bbbk[t][x; \sigma_{q, r}, \delta_{p(t)}]$ only in the following three cases:
    \begin{enumerate}
        \item [\rm (a)] $q = 1, r = 0$ with no restriction on $p(t)$;
        
        \item [\rm (b)] $q = 1, r\neq 0$ and $p(t) = c$, $c\in \Bbbk$;
        
        \item [\rm (c)] $q\neq 1, p(t) = c\left( t + \frac{r}{q-1} \right)$, $c\in \Bbbk$ with no restriction on $r$.
    \end{enumerate}
    
In any of the cases {\rm (a) - (c)} we have that $\nu_x \circ \nu_t = \nu_t \circ \nu_x$. If the Ore extension $\Bbbk[t][x; \sigma_{q, r}, \delta_{p(t)}]$ satisfies one of these three conditions, then it is differentially smooth \cite[Proposition 3.3]{Brzezinski2015}. 

As we saw in Proposition  \ref{BiquadraticalgebrasPBWbasisTwogenerators}, up to isomorphism, there are only five bi-quadratic algebras on two generators over $\Bbbk$, so that having in mind Brzezi{\'n}ski's result on the differential smoothness of $\Bbbk[t][x; \sigma_{q, r}, \delta_{p(t)}]$, we get that the algebras $\Bbbk[x_1,x_2],\ A_1(\Bbbk)$ and $U(\mathfrak{n}_2)$ belong to the case (a), while $\mathcal{O}_q(\Bbbk)$ belongs to the case (c), whence these four algebras are differentially smooth. With respect to the quantum Weyl algebra $A_{1}^{q}(\Bbbk)$, this cannot be expressed by using this kind of skew polynomial rings, so that its differential smoothness should be investigated by considering other and alternative approach. Precisely, Rubiano and Reyes \cite[Proposition 3.1]{RubianoReyes2024DSBiquadraticAlgebras} proved that $A_{1}^{q}(\Bbbk)$ is differentially smooth.

\subsection{Bi-quadratic algebras on three generators}\label{DSBiquadraticalgebrasThreegenerators}

Rubiano and Reyes \cite{RubianoReyes2024DSBiquadraticAlgebras} consider the bi-quadratic algebras in three generators and seek conditions under which these algebras are differentially smooth. In this context, they find additional relations beyond those provided by Bavula (Proposition \ref{Bavula2023Theorem1.3}) for the PBW basis. Nevertheless, several of the relations he introduced also appear as conditions for differential smoothness, as the following proposition shows. 

\begin{proposition}[{\cite[Theorem 3.2]{RubianoReyes2024DSBiquadraticAlgebras}}]\label{Firsttheoremsmoothnessbi-3quadraticalgebras}
If the conditions
\begin{align}
    c = \beta = \lambda = &\ 0, \\
    b_1(q_1-1)-ab = &\ 0, \\
    b_2(q_2-1)-\alpha\gamma = &\ 0, \\
    b_3(q_3-1) - \mu \nu = &\ 0, \quad {\rm and} \\
    \mu a = &\ 0
\end{align}

hold, then $A$ is differentially smooth.
\end{proposition}

Following Bavula's classification presented in \cite{Bavula2023}, Proposition \ref{Firsttheoremsmoothnessbi-3quadraticalgebras} shows that there are twenty two bi-quadratic algebras on three generators that are differentially smooth.

Rubiano and Reyes also provided necessary conditions for a bi-quadratic algebra on three variables to fail to be differentially smooth. It can be seen that these conditions depend on the constants $c$, $\beta$, and $\lambda$.

\begin{proposition}[{\cite[Theorem 3.3]{RubianoReyes2024DSBiquadraticAlgebras}}]\label{Secondtheoremsmoothnessbi-3quadraticalgebras}
If one of the conditions $c\not = 0$, $\beta\not=0$ or $\lambda\not=0$ holds, then $A$ is not differentially smooth.
\end{proposition}

Proposition \ref{Secondtheoremsmoothnessbi-3quadraticalgebras} shows that twenty one bi-quadratic algebras cannot be differentially smooth.

\subsection{Bi-quadratic algebras on four generators}\label{DSBiquadraticalgebrasFourgenerators}

Bavula \cite[p. 699]{Bavula2023} asserted that a construction of bi-quadratic algebras on four generators was introduced by Zhang and Zhang \cite{ZhangZhang2008, ZhangZhang2009} with their class of {\em double Ore extensions}. As one can appreciate in the literature, these extensions are of great importance in the noncommutative algebraic geometry, and more exactly, in the classification of {\em Artin-Schelter regular algebras} introduced by Artin and Schelter \cite{ArtinSchelter1987} (see Bellamy et al. \cite{Bellamyetal2016} and Rogalski \cite{Rogalski2023}). The study of the differential smoothness of a subclass of algebras of double Ore extensions known as {\em double extension regular algebras of type} {\rm (}14641{\rm )} has been carried out by Rubiano and Reyes \cite[Theorem 4.1]{RubianoReyes2024DSDoubleOreExtensions}.

\subsection{Bi-quadratic algebras on \texorpdfstring{$n$}{Lg} generators}\label{DSBiquadraticalgebrasngenerators}

Given the work carried out in the three-variable case, it is natural to ask what happens in the case of $n$ variables. Therefore, we focus on the conditions under which a bi-quadratic algebra is differentially smooth, assuming that it already admits a PBW basis.

Throughout this section, $A$ denotes a bi-quadratic algebra on $n$ generators denoted by $x_1, \ldots, x_n$ with PBW basis, subject to the relations {\rm (}\ref{Bavula2023(1)}{\rm )} and {\rm (}\ref{Bavula2023(2)}{\rm )}. Since these are a subclass of the {\em skew PBW extensions} introduced by Gallego and Lezama \cite{GallegoLezama2010}, it follows from \cite[Theorems 14 and 18]{Reyes2013} that its Gelfand-Kirillov dimension is $n$. This fact is key in Theorems \ref{Firsttheoremsmoothnessbi-quadraticalgebras} and \ref{Secondtheoremsmoothnessbi-quadraticalgebras}.

The following theorem is the first important result of the paper. We extend Brzezi\'nski's ideas \cite{Brzezinski2015} mentioned in Section \ref{DSBiquadraticalgebrasTwogenerators}.

\begin{theorem}\label{Firsttheoremsmoothnessbi-quadraticalgebras}
Let $1\leq i <j \leq n$ and $1\leq k \leq n$. If the conditions
\begin{align}
    a_{ij,k} = &\ 0, \text{\rm for } k\ne i,j, \\
    b_{ij}(q_{ij}-1)-a_{ij,i}a_{ij,j} = &\ 0, \\
    a_{ij,j}(1-q_{ik})= &\ a_{ik,k}(1-q_{ij}), \text{\rm for } i<k<j, \\
    a_{kj,j}(1-q_{ik})= &\ a_{ik,i}(1-q_{kj}), \label{condijk11} \text{\rm for } i<k<j, \\
    a_{ij,j}a_{kj,j}+b_{kj}(q_{kj}-q_{ij})= &\ 0, \text{\rm for } i<k<j, \\
    b_{ij}(q_{ik}-q_{kj})+a_{kj,k}a_{ij,j}q_{ik}-q_{kj}a_{ik,k}a_{ij,i}= &\ 0, \text{\rm 
 for } i<k<j, \\
 a_{ij,i}(1-q_{kj})= &\ a_{kj,k}(1-q_{ij}), \text{\rm for } k<i<j, \\
 a_{ki,i}(1-q_{kj})= &\ a_{kj,j}(1-q_{ki}), \label{condijk21} \text{\rm for } k<i<j, \\
 a_{ij,j}(1-q_{ki})= &\ a_{ki,k}(1-q_{ij}), \text{\rm for } k<i<j, \\
 b_{ij}(1-q_{ki}q_{kj})+a_{ki,k}a_{ij,i}+q_{ki}a_{ij,j}a_{kj,k}= &\ 0, \text{\rm for } k<i<j, \\
 a_{jk,k}(1-q_{ij})= &\ a_{ij,i}(1-q_{jk}), \text{\rm for } i<j<k, \\
 a_{ik,i}(1-q_{jk})= &\ a_{jk,j}(1-q_{ik}), \label{condijk31} \text{\rm for } i<j<k, \\
 a_{ik,k}(1-q_{ij})= &\ a_{ij,j}(1-q_{ik}), \text{\rm for } i<j<k, \\
 b_{ij}(1-q_{ik}q_{jk})+a_{ik,k}a_{ij,i}+q_{ik}a_{ij,j}a_{jk,k}= &\ 0, \text{\rm for } i<j<k
\end{align}

hold, then $A$ is differentially smooth.
\end{theorem}
\begin{proof}
Consider the following automorphisms for $1 \leq i<j \leq n$:
\begin{align}
   \nu_{x_i}(x_i) = &\ x_i, & \nu_{x_i}(x_j) = &\ q_{ij}^{-1}(x_j-a_{ij,i}),\quad {\rm and} \label{Auto1} \\ 
    \nu_{x_j}(x_i) = &\ q_{ij}x_i+a_{ij,j}, & \nu_{x_j}(x_j) = &\ x_j. \label{Auto2} 
\end{align}

The map $\nu_{x_i}$ can be extended to an algebra homomorphism of $A$ if and only if the definitions of $\nu_{x_i}(x_i)$ and for $1 \leq i<j \leq n$, i.e.
\begin{align*}
   \nu_{x_i}(x_i)\nu_{x_i}(x_j)-q_{ij}\nu_{x_i}(x_j)\nu_{x_i}(x_i) =  a_{ij,i}\nu_{x_i}(x_i) + a_{ij,j}\nu_{x_i}(x_j)+b_{ij}.
\end{align*}

We obtain the equations 
\begin{align*}
     b_{ij}(q_{ij}-1)-a_{ij,i}a_{ij,j} =  0, \text{\rm 
 for } 1 \leq i<j \leq n.
\end{align*}

Similarly, the map $\nu_{x_j}$ can be extended to an algebra homomorphism of $A$ if and only if the equalities
\begin{align*}
     b_{ij}(q_{ij}-1)-a_{ij,i}a_{ij,j} =  0, \text{\rm 
 for } 1 \leq i<j \leq n.
\end{align*}

are satisfied. 

Now, let $1\leq k \leq n$ and we consider the extension of the map $\nu_{x_k}$ to an algebra homomorphism of $A$. Here, we must consider three cases:
\begin{itemize}
    \item $i<k<j$. In this case, the map $\nu_{x_k}$ can be extended to an algebra homomorphism of $A$ if and only if the equalities
    \begin{align*}
        a_{ij,j}(1-q_{ik})= &\ a_{ik,k}(1-q_{ij}), \\
        a_{ij,j}a_{kj,j}+b_{kj}(q_{kj}-q_{ij})= &\ 0  \text{ and } \\
        b_{ij}(q_{ik}-q_{kj})+a_{kj,k}a_{ij,j}q_{ik}-q_{kj}a_{ik,k}a_{ij,i}= &\ 0, 
    \end{align*}
    
    are satisfied.
    \item $k<i<j$. Now, the map $\nu_{x_k}$ can be extended to an algebra homomorphism of $A$ if and only if the equalities
    \begin{align*}
         a_{ij,i}(1-q_{kj})= &\ a_{kj,k}(1-q_{ij}), \\
        a_{ij,j}(1-q_{ki})= &\ a_{ki,k}(1-q_{ij}) \text{ and }  \\
        b_{ij}(1-q_{ki}q_{kj})+a_{ki,k}a_{ij,i}+q_{ki}a_{ij,j}a_{kj,k}= &\ 0, 
    \end{align*}
    
    hold.
    
    \item $i<j<k$. In this last case, the map $\nu_{x_k}$ can be extended to an algebra homomorphism of $A$ if and only if the equalities
    \begin{align*}
         a_{jk,k}(1-q_{ij})= &\ a_{ij,i}(1-q_{jk}), \\
         a_{ik,k}(1-q_{ij})= &\ a_{ij,j}(1-q_{ik}) \text{ and }  \\
         b_{ij}(1-q_{ik}q_{jk})+a_{ik,k}a_{ij,i}+q_{ik}a_{ij,j}a_{jk,k}= &\ 0, 
    \end{align*}
    
    are satisfied.
\end{itemize}

Since we need to guarantee that
\begin{align}
   \nu_{x_i} \circ \nu_{x_j} = &\ \nu_{x_j} \circ \nu_{x_i}, \label{compij} \\
   \nu_{x_i} \circ \nu_{x_k} = &\  \nu_{x_k} \circ \nu_{x_i}, \label{compik} \quad {\rm and} \\
   \nu_{x_j} \circ \nu_{x_k} = &\ \nu_{x_k} \circ \nu_{x_j}, \label{compjk} 
\end{align}

for $1\leq i,j,k \leq n$, $i<j$, it is enough to satisfy these equalities for the generators $x_i$, $x_j$ and $x_k$. By using {\rm (}\ref{compij}{\rm )}, we have that 
\begin{align}
\nu_{x_i} \circ \nu_{x_j}(x_i) = &\ q_{ij}x_i+a_{ij,j}, \\
\nu_{x_j} \circ \nu_{x_i}(x_i) = &\ q_{ij}x_i+a_{ij,j}, \\ 
\nu_{x_i} \circ \nu_{x_j}(x_j) = &\ q_{ij}^{-1}(x_j-a_{ij,i}), \\
\nu_{x_j} \circ \nu_{x_i}(x_j) = &\ q_{ij}^{-1}(x_j-a_{ij,i}), \\
\nu_{x_i} \circ \nu_{x_j}(x_k) = &\ q_{kj}q_{ik}^{-1}(x_k-a_{ik,i})+a_{kj,j}, \label{relijk11}\text{ for } i<k<j, \\
\nu_{x_j} \circ \nu_{x_i}(x_k) = &\ q_{ik}^{-1}(q_{kj}x_k+a_{kj,j}-a_{ik,i}), \label{relijk12} \text{ for } i<k<j, \\
\nu_{x_i} \circ \nu_{x_j}(x_k) = &\ q_{kj}(q_{ki}x_k+a_{ki,i})+a_{kj,j}, \label{relijk21}\text{ for } k<i<j, \\
\nu_{x_j} \circ \nu_{x_i}(x_k) = &\ q_{ki}(q_{kj}x_k+a_{kj,j})+a_{ki,i}, \label{relijk22} \text{ for } k<i<j, \\
\nu_{x_i} \circ \nu_{x_j}(x_k) = &\ q_{jk}^{-1}(q_{ik}^{-1}(x_k-a_{ik,i})-a_{jk,j}), \label{relijk31} \text{ for } i<j<k, \text{ and } \\
\nu_{x_j} \circ \nu_{x_i}(x_k) = &\ q_{ik}^{-1}(q_{jk}^{-1}(x_k-a_{jk,j})-a_{ik,i}), \label{relijk32} \text{ for } i<j<k.
\end{align}

As we can see, composition $\nu_{x_i}\circ\nu_{x_j}=\nu_{x_j}\circ\nu_{x_i}$ for the generators $x_i$ and $x_j$. Relations {\rm (}\ref{relijk11}{\rm )} and {\rm (}\ref{relijk12}{\rm )} coincide when $a_{kj,j}(1-q_{ik})=  a_{ik,i}(1-q_{kj})$; this is exactly the condition {\rm (}\ref{condijk11}{\rm )}. Now, relations {\rm (}\ref{relijk21}{\rm )} and {\rm (}\ref{relijk22}{\rm )} are equal if $a_{ki,i}(1-q_{kj})=  a_{kj,j}(1-q_{ki})$. This holds due to the condition {\rm (}\ref{condijk21}{\rm )}. Also, relations {\rm (}\ref{relijk31}{\rm )} and {\rm (}\ref{relijk32}{\rm )} coincide when $ a_{ik,i}(1-q_{jk})= a_{jk,j}(1-q_{ik})$. This is true because of the condition {\rm (}\ref{condijk31}{\rm )}.

Next, 
\begin{align}
    \nu_{x_i} \circ \nu_{x_k}(x_i) = &\  \nu_{x_i}( \nu_{x_k}(x_i)),\label{relikk11}\text{ for all } 1\leq k \leq n, \\
    \nu_{x_k} \circ \nu_{x_i}(x_i) = &\  \nu_{x_k}(x_i), \label{relikk12} \text{ for all } 1\leq k \leq n, \\ 
    \nu_{x_i} \circ \nu_{x_k}(x_j) = &\ q_{kj}^{-1}(q_{ij}^{-1}(x_j-a_{ij,i})-a_{kj,k}),\label{relikk21} \text{ for } k<j, \\ 
    \nu_{x_k} \circ \nu_{x_i}(x_j) = &\ q_{ij}^{-1}(q_{kj}^{-1}(x_j-a_{kj,j})-a_{ij,i}),\label{relikk22} \text{ for } k<j, \\ 
    \nu_{x_i} \circ \nu_{x_k}(x_j) = &\ q_{jk}q_{ij}^{-1}(x_j-a_{ij,i})+a_{jk,k},\label{relikk41} \text{ for } j<k, \\ 
    \nu_{x_k} \circ \nu_{x_i}(x_j) = &\ q_{ij}^{-1}(q_{jk}x_j+a_{jk,k}-a_{ij,i}),\label{relikk42} \text{ for } j<k, \\
    \nu_{x_i} \circ \nu_{x_k}(x_k) = &\ \nu_{x_i}(x_k), \label{relikk51}\text{ for all } 1\leq k \leq n, \\
    \nu_{x_k} \circ \nu_{x_i}(x_k) = &\ \nu_{x_k}(\nu_{x_i}(x_k)), \label{relikk52} \text{ for all } 1\leq k \leq n.
\end{align}

Again, the compositions shown in {\rm (}\ref{relikk11}{\rm )} and {\rm (}\ref{relikk12}{\rm )} are equal regardless of the value of $k$ since $\nu_{x_i}$ acts as the identity on $x_i$ and $\nu_{x_k}$ is linear. 

Relations {\rm (}\ref{relikk21}{\rm )} and {\rm (}\ref{relikk22}{\rm )} hold if $a_{kj,k}(1-q_{ij})=a_{ij,i}(1-q_{kj})$; and relations {\rm (}\ref{relikk41}{\rm )} and {\rm (}\ref{relikk42}{\rm )} are equal when $a_{jk,k}(1-q_{ij})=a_{ij,i}(1-q_{jk})$.

The compositions shown in {\rm (}\ref{relikk51}{\rm )} and {\rm (}\ref{relikk52}{\rm )} coincide for the same reason of the composition of $\nu_{x_i} \circ \nu_{x_k}(x_i) =\nu_{x_k} \circ \nu_{x_i}(x_i)$.

Finally, 
\begin{align}
    \nu_{x_j} \circ \nu_{x_k}(x_i) = &\ q_{ki}^{-1}(q_{ij}x_i+a_{ij,j}-a_{ki,k}) ,\label{reljkk11}\text{ for } k<i, \\
    \nu_{x_k} \circ \nu_{x_j}(x_i) = &\ q_{ij}q_{ki}^{-1}(x_i-a_{ki,k})+a_{ij,j} , \label{reljkk12}\text{ for } k<i,  \\
    \nu_{x_j} \circ \nu_{x_k}(x_i) = &\ q_{ik}(q_{ij}x_i+a_{ij,j})+a_{ik,k},\label{reljkk41} \text{ for } i<k, \\ 
    \nu_{x_k} \circ \nu_{x_j}(x_i) = &\ q_{ij}(q_{ik}x_i+a_{ik,k})+a_{ij,j},\label{reljkk42} \text{ for } i<k, \\
    \nu_{x_j} \circ \nu_{x_k}(x_j) = &\ \nu_{x_j}( \nu_{x_k}(x_j)),\label{reljkk21}\text{ for all } 1\leq k \leq n, \\ 
    \nu_{x_k} \circ \nu_{x_j}(x_j) = &\ \nu_{x_k}(x_j),\label{reljkk22}\text{ for all } 1\leq k \leq n, \\ 
    \nu_{x_j} \circ \nu_{x_k}(x_k) = &\ \nu_{x_j}(x_k), \label{reljkk51}\text{ for all } 1\leq k \leq n, \\
    \nu_{x_k} \circ \nu_{x_j}(x_k) = &\ \nu_{x_k}(\nu_{x_j}(x_k)), \label{reljkk52} \text{ for all } 1\leq k \leq n.
\end{align}

Relations {\rm (}\ref{reljkk11}{\rm )} and {\rm (}\ref{reljkk12}{\rm )} coincide $a_{ki,k}(1-q_{ij})=a_{ij,j}(1-q_{ki})$; and relations {\rm (}\ref{reljkk41}{\rm )} and {\rm (}\ref{reljkk42}{\rm )} are equal if $a_{ik,k}(1-q_{ij})=a_{ij,j}(1-q_{ik})$.

Now, the compositions shown in {\rm (}\ref{reljkk21}{\rm )} and {\rm (}\ref{reljkk22}{\rm )} are equal regardless of the value of $k$, since $\nu_{x_j}$ acts as the identity on $x_i$ and $\nu_{x_k}$ is linear. And the same with relations {\rm (}\ref{reljkk51}{\rm )} and {\rm (}\ref{reljkk52}{\rm )}.

Consider $\Omega^{1}A$ a free right $A$-module of rank $n$ with generators $dx_i$, $1 \leq i \leq n$. For all $p\in A$ define a left $A$-module structure by
\begin{align}
    pdx_i = &\ dx_i\nu_{x_i}(p) \label{relrightmod}.
\end{align}

The relations in $\Omega^{1}A$ are given by 
\begin{align}
x_idx_i = &\ dx_i x_i, \notag \\
x_idx_j = &\ q_{ij}dx_jx_i+a_{ij,j}dx_j, \text{ and } \label{rel1}  \\
x_jdx_i = &\ q_{ij}^{-1}dx_ix_j-q_{ij}^{-1}a_{ij,j}dx_i,  \notag
\end{align}
for $1\leq i < j \leq n$.

We want to extend the correspondences 
\begin{equation*}
x_i \mapsto d x_i,
\end{equation*} 

to a map $d: A \to \Omega^{1} A$ satisfying the Leibniz's rule, for $1\leq i \leq n $. This is possible if it is compatible with the nontrivial relations {\rm (}\ref{Bavula2023(1)}{\rm )} and {\rm (}\ref{Bavula2023(2)}{\rm )}, i.e. if the equalities
\begin{align*}
        dx_ix_j+x_idx_j &\ = q_{ij}dx_jx_i+q_{ij}x_jdx_i+a_{ij,i}dx_i+a_{ij,j}x_j, \\
\end{align*}

hold for $1\leq i \leq n$. Note that $d(b_{ij})=0$ for all $1\leq i < j \leq n$.

Define $\Bbbk$-linear maps 
\begin{equation*}
\partial_{x_i}: A \rightarrow A
\end{equation*}

such that
\begin{align*}
    d(a)=\sum_{k=1}^{n}dx_k\partial_{x_k}(a), \quad {\rm for\ all} \ a \in A.
\end{align*}

Since $dx_i$ are free generators of the right $A$-module $\Omega^1A$, for $1\leq i \leq n$, these maps are well-defined. Note that $d(a)=0$ if and only if $\partial_{x_i}(a)=0$, for $1\leq i \leq n$. By using the relation in {\rm (}\ref{relrightmod}{\rm )} and the definitions of the maps $\nu_{x_i}$ and $\nu_{x_j}$, we get that
\begin{align}
\partial_{x_k}(x_1^{l_1}\cdots x_n^{l_n}) =l_k\prod_{r=1}^{k-1}(q_{rk}x_r+a_{rk,k})^{l_r}x_k^{l_k-1}x_{k+1}^{l_{k+1}}\cdots x_n^{l_n} &\ ,
\end{align}

where $1\leq k \leq n$ and $l_1, \ldots, l_n \in \mathbb{N}$.

In this way, $d(a)=0$ if and only if $a$ is a scalar multiple of the identity. This shows that $(\Omega A,d)$ is connected where $\Omega A = \bigoplus_{r=0}^{n-1} \Omega^r A$.

The universal extension of $d$ to higher forms compatible with {\rm (}\ref{rel1}{\rm )} gives the following rules for $\Omega^lA$  $(2\leq l\leq n-1)$:
\begin{align}
\bigwedge_{k=1}^{l}dx_{q(k)} = &\ (-1)^{\sharp}\prod_{(s,t)\in P}q_{st}^{-1}\bigwedge_{k=1}^ldx_{p(k)}
\end{align}

where 
$$
q:\{1,\ldots,l\}\rightarrow \{1,\ldots,n\}
$$ 

is an injective map and 
$$
p:\{1,\ldots,l\}\rightarrow \text{Im}(q)$$

is an increasing injective map and $\sharp$ is the number of $2$-permutations needed to transform $q$ into $p$, and 
$$
P := \{(s, t) \in \{1, \ldots, l\} \times \{1, \ldots, l\} \mid q(s) >  q(t)\}.
$$

Since the automorphisms $\nu_{x_i}$, $1\leq i \leq n$ commute with each other, there are no additional relations to the previous ones, so we get that
\begin{align*}
\Omega^{n-1}A = &\ \left[\bigoplus_{r=1}^{n-1}dx_1\wedge \cdots dx_{r-1}\wedge dx_{r+1} \wedge \cdots \wedge dx_{n} \right]A.
\end{align*}

Since $\Omega^nA = \omega A\cong A$ as a right and left $A$-module, with $\omega=dx_1\wedge \cdots \wedge dx_n$, where $\nu_{\omega}=\nu_{x_1}\circ\cdots\circ\nu_{x_n}$, we have that $\omega$ is a volume form of $A$. From Proposition \ref{BrzezinskiSitarz2017Lemmas2.6and2.7} (2) we get that $\omega$ is an integral form by setting
\begin{align*}
    \omega_i^j = &\ \bigwedge_{k=1}^{j}dx_{p_{i,j}(k)}, \text{ for } 1\leq i \leq \binom{n}{j-1}, \\
    \bar{\omega}_i^{n-j} = &\ (-1)^{\sharp_{i,j}}\prod_{(r,s)\in P_{i,j}}q_{rs}^{-1}\bigwedge_{k=j+1}^{n}dx_{\bar{p}_{i,j}(k)}, \text{ for } 1\leq i \leq \binom{n}{j-1},
\end{align*}
for $1\leq j \leq n$ and where 
\begin{align*}
    p_{i,j}:\{1,\ldots,j\}\rightarrow &\ \{1,\ldots,n\}, \quad {\rm and} \\
\bar{p}_{i,j}:\{j+1,\ldots,n\}\rightarrow &\ (\text{Im}(p_{i,j}))^c
\end{align*}

(the symbol $\square^c$ denotes the complement of the set $\square$), are increasing injective maps, and $\sharp_{i,j}$ is the number of $2$-permutation needed to transform 
\[
\left\{\bar{p}_{i,j}(j+1),\ldots, \bar{p}_{i,j}(n), p_{i,j}(1), \ldots, p_{i,j}(j)\right\} \quad {\rm into\ the\ set} \quad \{1, \ldots, n\},
\]

and 
\[
P_{i,j} :=\{(r, s) \in \{1, \ldots, j\}\times\{j+1, \ldots, n\} \mid p_{i,j}(r)< \bar{p}_{i,j}(s)\}.
\]

Consider $\omega' \in \Omega^jA$, that is,  
\begin{align*}
\omega' =\sum_{i=1}^{\binom{n}{j-1}}\bigwedge_{k=1}^{j}dx_{p_{i,j}(k)}b_i, \quad {\rm with} \  n_i \in \Bbbk.
\end{align*}

Then
\begin{align*}
 \sum_{i=1}^{\binom{n}{j-1}}\omega_{i}^{j}\pi_{\omega}(\bar{\omega}_i^{n-j}\wedge \omega') = &\ \sum_{i=1}^{\binom{n}{j-1}}\left[\bigwedge_{k=1}^{j}dx_{p_i(k)}\right] \cdot  \pi_{\omega} \left[(-1)^{\sharp_{i,j}} \square^{*} \wedge \omega'\right] \\
 = &\ \displaystyle \sum_{i=1}^{\binom{n}{j-1}}\bigwedge_{k=1}^{j}dx_{p_{i,j}(k)}n_i =  \omega',
\end{align*}

where 
\begin{align*}
    \square^{*} := &\ \prod_{(r,s)\in P_{i,j}}q_{rs}^{-1} \bigwedge_{k=j+1}^{n}dx_{\bar{p}_{i,j}(k)}.
\end{align*}

By Proposition \ref{BrzezinskiSitarz2017Lemmas2.6and2.7} (2), it follows that $A$ is differentially smooth.
\end{proof}

We conclude the paper with the following result.

\begin{theorem}\label{Secondtheoremsmoothnessbi-quadraticalgebras}
If there exists $a_{ij,k}\not=0$ for some $1\leq i<j\leq n$, $k\ne i,j$, then $A$ is not differentially smooth.
\end{theorem}
\begin{proof}
By contradiction. Suppose that $A$ has a first order differential calculus $(\Omega A, d)$. We consider $a_{ij,k}\not =0$. Since $d$ must be compatible with the relation {\rm (}\ref{Bavula2023(2)}{\rm )}, then we get that
\begin{equation*}
dx_ix_j+x_idx_j=q_{ij}dx_jx_i+q_{ij}x_jdx_i+a_{ij,i}dx_i+a_{ij,j}dx_j+a_{ij,k}dx_k,
\end{equation*}

whence $dx_k$ is generated by $dx_i$ and $dx_j$. This means that $\Omega^1A$ is generated by $n-1$ elements and $\Omega^nA=\bigwedge_{k=1}^{n}\Omega^1A=0$, i.e. there is no $n$-order calculus. Since ${\rm GKdim}(A) = n$, we conclude that $A$ cannot be differentially smooth.
\end{proof}

\begin{remark}
The conditions appearing in Theorem \ref{Firsttheoremsmoothnessbi-quadraticalgebras} generalize those corresponding formulated in Proposition \ref{Firsttheoremsmoothnessbi-3quadraticalgebras}.
\end{remark}

\begin{remark}
It is important to note that if one algebra is not differential smooth, this does not mean that the algebra does not possess a differential calculus. The universal enveloping algebra $U(\mathfrak{sl}_2(\Bbbk))$ is an illustration of this fact as it was shown by Beggs and Majid \cite{BeggsMajid2020}. By using Hopf algebras, they proved the existence of a first-order differential calculus and other properties of Riemannian quantum geometry that are different from the differential smoothness shown in this paper.
\end{remark}

\section{Examples}\label{Examplesdifferentiallysmoothngenerators}

The following algebras are differentially smooth.

\begin{itemize}
    \item \emph{Commutative polynomial ring}: (Example \ref{examplesDSalgebrasBre}) $\Bbbk[x_1, \ldots, x_n]$ is a bi-quadratic algebra, where $q_{ij}=1$, $a_{ij,k}=b_{ij}=0$,  for $1\leq i<j\leq n$ and $1\leq k\leq n$.
    
    \item \emph{Quantum plane} (Section \ref{DSBiquadraticalgebrasTwogenerators}): In $\mathcal{O}_q(\Bbbk^2)$ we have the well-known relation $xy=qyx$ ($q\in \Bbbk^{*}$), with $q_{12} = q$ and $a_{12,1}=a_{12,2}=b_{12}=0$.
    
    \item The $n${\em th Weyl algebra}: $A_n(\Bbbk)$ is the $\Bbbk$-algebra generated by the $2n$ indeterminates $x_1,\dotsc, x_n, y_1,\dotsc, y_n$ where
    {\normalsize{\begin{align*}
    x_jx_i = &\ x_ix_j,\ \ \ \ \ y_jy_i = y_iy_j,\ \ \ \ \ \ \ \ \ \ \ \ \ \ \ \ \ 1\le i < j\le n,\notag \\
    y_jx_i = &\ x_iy_j + \delta_{ij},\ \ \ \ \delta_{ij}\ {\rm is\ the\ Kronecker}' {\rm s\
    delta},\ \ \ 1\le i, j \le n.
    \end{align*}}}
    
    If we identify $y_i=x_{n+i}$ for $1\leq i \leq n$, then  
    \begin{align*}
        q_{ij} = &\ 1, \quad 1\leq i<j<2n, \\
        a_{ij,k} = &\ 0, \quad 1\leq i,j,k \leq 2n, i<j, \\
        b_{ij} = &\ 0, j\ne n+i, \\
        b_{i(n+i)} = &\ 1, \quad 1\leq i \leq n.
    \end{align*}
    
    \item {\em Multiplicative analogue of the Weyl algebra}: $\mathcal{O}_{n}(\lambda_{ij})$ is generated over $\Bbbk$ by the indeterminates $x_1, \dotsc, x_n$ subject to the relations
    \[
    x_j x_i = \lambda_{ij} x_i x_j, \quad {\rm for}\ 1\le i < j \le n \ {\rm and}\ \lambda_{ij}\in \Bbbk^{*}.
    \]
    
    where $q_{ij}=\lambda_{ij}$, $a_{ij,k}=b_{ij}=0$ for $1\leq i<j\leq n$ and $1\leq k\leq n$.
    
    \item {\em Algebra of linear partial shift operators}: we consider the $\Bbbk$-algebra of linear partial shift operators $A = \Bbbk[t_1, \ldots, t_n][E_1, \ldots, E_m]$ subject to the relations
    \begin{align*}
        t_jt_i &\ =t_it_j, \quad 1\leq i<j  \leq n, \\
        E_it_i &\ = t_iE_i+E_i, \quad 1 \leq i  \leq n, \\
        E_jt_i &\ = t_iE_j, \quad i \ne j, \\
        E_jE_i &\ =E_iE_j, \quad 1 \leq i<j\leq m.
    \end{align*}
    
     If we identify the generators as
    \begin{align*}
        x_i &\ =t_i, \quad 1\leq i \leq n, \\
        x_{n+j} &\ = E_j, \quad 1\leq j \leq m,
    \end{align*}

    then 
     \begin{align*}
        q_{ij} = &\ 1, \quad 1\leq i<j<n+m, \\
        a_{ij,k} = &\ 0, \quad 1\leq i,j,k \leq 2n, i<j, j\ne m+i,  \\
        a_{i(m+i),k} = &\ 0, \text{ if } k\ne m+i, \\
        a_{i(m+i),m+i} = &\ 1, \quad 1\leq i \leq n, \\
        b_{ij} = &\ 0, \quad 1\leq i<j<n+m.
    \end{align*}
    
    \item {\em Algebra of linear partial difference operators}: we consider the $\Bbbk$-algebra of linear partial difference $B = \Bbbk[t_1, \ldots, t_n][\Delta_1, \ldots, \Delta_m]$ subject to the relations
    \begin{align*}
        t_jt_i &\ =t_it_j, \quad 1\leq i<j  \leq n, \\
        \Delta_it_i &\ = t_i\Delta_i+ \Delta_i+1, \quad 1 \leq i  \leq n, \\
        \Delta_jt_i &\ = t_i\Delta_j, \quad i \ne j, \\
        \Delta_j\Delta_i &\ =\Delta_i\Delta_j, \quad 1 \leq i<j\leq m.
    \end{align*}
    
    If we identify the generators as the algebra of linear partial shift operators above, then we get that
    \begin{align*}
        b_{ij} &\ =0, \text{ if } j\ne m+i, \\b_{i(m+i)} &\ =1, \quad 1\leq i \leq n.
    \end{align*}
\end{itemize}

The following algebras satisfy the conditions formulated in Theorem \ref{Secondtheoremsmoothnessbi-quadraticalgebras}, so that all are not differentially smooth.
\begin{itemize}
\item {\em $q$-Heisenberg algebra}: {\bf H}$_n(q)$ is the $\Bbbk$-algebra generated over $\Bbbk$ by the indeterminates $x_i, y_i, z_i$, for $1\le i\le n$, subject to the relations 
{\normalsize{\begin{align*}
    x_ix_j = &\ x_jx_i,\ \ \ y_iy_j = y_jy_i,\ \ \ z_jz_i = z_iz_j, &\ 1\le i < j\le n,\notag \\
    x_i z_i  - qz_ix_i = &\ z_iy_i - qy_iz_i = x_iy_i - q^{-1}y_ix_i + z_i = 0, &\ 1\le i\le n, \notag \\
    x_iy_j = &\ y_jx_i,\ \ \ x_iz_j = z_jx_i,\ \ \ y_iz_j = z_jy_i, &\ i\neq j.
\end{align*}}}
    
If $y_i := x_{n+i}$, $z_i := x_{2n+i}$ with $1 \leq i \leq n$, then we get that
\begin{align*}
x_ix_{n+i}=q^{-1}x_{n+i}x_i-x_{2n+i}, \quad 1\leq i \leq n.
\end{align*}

In this way, $a_{i(n+i),2n+1}=-1$ for $1 \leq i \leq n$. 

\item In Example \ref{Examplesbi-quadraticthreegenerators} (c), we have that $a_{12,3} = -q^{-1/2}$.

\item {\em Askey-Wilson algebra} (Example \ref{Examplesbi-quadraticthreegenerators} (d)): If we identify $K_0=x_1$, $K_1 = x_2$ and $K_2 = x_3$, then $a_{12,3}=e^{-\omega}$.

\item {\em Dispin algebra}: $U(\mathfrak{osp}(1, 2))$. This $\Bbbk$-algebra is generated by the indeterminates $x_1, x_2$ and $x_3$ subject to the relations
\[
    x_1 x_2 - x_2 x_1 = 1,\quad x_3x_1 + x_1 x_3 = x_2 \quad {\rm and}\quad x_2x_3 - x_3x_2 = x_3,
\]

whence $a_{13,2} = 1$.

\item  Let $U(\mathfrak{sl}(2, \Bbbk))$ be the universal enveloping algebra of $\mathfrak{sl}(2, \Bbbk)$. By definition, this $\Bbbk$-algebra is generated by the indeterminates $x, y, z$ subject to the relations
\[
    xy - yx = z,\quad xz - zx = -2x \quad {\rm and}\quad yz - zy = 2y.
\]

As it is clear, $a_{12,3} = 1$ where $x=x_1$, $y=x_2$ and $z=x_3$.

\item  Let $U(\mathfrak{so}(3, \Bbbk))$ be the universal enveloping algebra of $\mathfrak{so}(3, \Bbbk)$, that is, the $\Bbbk$-algebra generated by the indeterminates $x, y, z$ subject to the relations
\[
    xy - yx = z, \quad xz - zx = -y \quad {\rm and} \quad yz - zy = x.
\]

As it can be seen, $a_{23,1}=1$ where $x=x_1$, $y=x_2$ and $z=x_3$.

\item The quantum universal enveloping algebra $U'(\mathfrak{so}(3, \Bbbk))$ with $q\in \Bbbk^{*}$ is defined as the $\Bbbk$-algebra generated by $x, y, z$ and relations
\[
    yx - qxy = -q^{1/2}z, \quad zx - q^{-1}xz = q^{-1/2} y \quad {\rm and} \quad zy - qyz = -q^{1/2} x.
\]

In this way, we have that $a_{12,3} = -q^{1/2}$, where $x=x_1$, $y=x_2$ and $z=x_3$. 

\item The {\em Woronowicz algebra}: $W_{v} (\mathfrak{sl}(2, \Bbbk))$, where $v \in \Bbbk - \{0\}$ is not a root of unity, is the $\Bbbk$-algebra generated by the indeterminates $x, y, z$ and relations given by
    \[
    xz - v^4 zx = (1 + v^2)x, \quad xy - v^2yx = v z \quad {\rm and}\quad zy - v^4yz = (1 + v^2)y. 
    \]

It follows that $a_{12,3}=v$, where $x=x_1$, $y=x_2$ and $z=x_3$.
\end{itemize}

\begin{remark}\label{CQWABavula2024}
The $3$-{\em cyclic quantum Weyl algebras} $A(\alpha, \beta, \gamma;q^2)$, $\alpha, \beta, \gamma, q^2 \in \Bbbk$, where $q^2$ is not a root of unity, investigated by Bavula \cite{Bavula2024} are defined as 
\begin{align*}
    xy= &\ q^2yx+\alpha , \\
    xz = &\ q^{-2}zx+\beta, \text{and} \\
    yz= &\ q^2zy+\gamma.
\end{align*}

In his paper, he studied their prime spectra and a classification of their simple modules. 

As we can see from the definition, if $x := x_1, y := x_2$ and $z := x_3$, then we obtain that 
\begin{gather}\label{conditions3CQWA}
    q_{12}=q^2, \ q_{13}=q^{-2}, \ q_{23}=q^2, \ b_{12}= \alpha, \ b_{13}=\beta, \ b_{23}= \gamma, \\
    a_{ij,k}=0,  \ i,j ,k \in \{1,2,3\}.
\end{gather}

whence by Proposition \ref{Bavula2023Theorem1.3} it follows that $A(\alpha, \beta, \gamma;q^2)$ are bi-quadractic algebras with PBW basis.

As can be seen, conditions \ref{conditions3CQWA} do not satisfy the assumptions established in Theorems \ref{Firsttheoremsmoothnessbi-quadraticalgebras} and \ref{Secondtheoremsmoothnessbi-quadraticalgebras}. Nevertheless, if we consider the family of automorphisms of $A(\alpha, \beta, \gamma)$ given by
\begin{align*}
   \nu_{x}(x) = &\ q^{-2}x, & \nu_{x}(y) = &\ q^{2}y, & \nu_{x}(x_3) = &\ q^{-2}z,  \\ 
    \nu_{y}(x) = &\ q^{-2}x, & \nu_{y}(y) = &\ q^{2}y, &  \nu_{y}(z) = &\ q^2z,  \\
    \nu_{z}(x) = &\ q^{2}x, & \nu_{z}(y) = &\ q^{-2}y, & \nu_{z}(z) = &\ q^{-2}z, 
\end{align*}

it is straightforward to see that these automorphisms satisfy Leibniz's rule and commute with each other. A similar argument to the presented above in the proof of Theorem \ref{Firsttheoremsmoothnessbi-quadraticalgebras} allows us to assert that the $3$-cyclic quantum Weyl algebras are differentially smooth.
\end{remark}

\section{Future work}\label{FutureworkDifferentialsmoothnessofbiquadraticalgebras}

Since bi-quadratic algebras are related to other families of noncommutative rings of polynomial type such as those mentioned in the Introduction, a second natural task is to investigate the extension of the theory developed here to these families of algebras with the aim of studying its differential smoothness.

\end{document}